\documentclass[12pt,draft]{amsart}
\usepackage{amsmath, amsthm, amscd, amsfonts, amssymb, graphicx, color}
\usepackage[bookmarksnumbered, colorlinks, plainpages]{hyperref}
\usepackage{enumerate}
\hypersetup{colorlinks=true,linkcolor=red, anchorcolor=green, citecolor=cyan, urlcolor=red, filecolor=magenta, pdftoolbar=true}
\usepackage{epstopdf}

\newtheorem{thm}{Theorem}[section]
\newtheorem{cor}[thm]{Corollary}
\newtheorem{lem}[thm]{Lemma}
\newtheorem{prop}[thm]{Proposition}
\newtheorem{conj}[thm]{Conjecture}

\theoremstyle{definition}

\theoremstyle{remark}

\numberwithin{equation}{section}

\begin{document}

\setcounter{page}{1}

\title[On the Linear Independence of Finite Wavelet Systems ]{On the Linear Independence of Finite Wavelet Systems  Generated by Schwartz Functions or Functions with certain behavior at infinity }

\author[Abdelkrim Bourouihiya]{Abdelkrim Bourouihiya\\
\today }
\address{Department of Mathematics, Nova Southeastern University}
\curraddr{3301 College Avenue, Fort Lauderdale, Florida, USA}
\email{\textcolor[rgb]{0.00,0.00,0.84}{ab1221@nova.edu}}

%\dedicatory{This paper is dedicated to Professor ABCD}

%\let\thefootnote\relax\footnote{Copyright 2016 by the Tusi Mathematical Research Group.}



%\date{Received: xxxxxx; Revised: yyyyyy; Accepted: zzzzzz.
%\newline \indent $^{*}$Corresponding author}

\begin{abstract}
One of the motivations to state HRT conjecture on the linear independence of finite Gabor systems was the fact that there are  linearly dependent Finite Wavelet Systems (FWS). Meanwhile, there are also many examples of linearly independent FWS, some of which are presented in this paper. We prove the linear independence of every three point FWS generated by a nonzero Schwartz function  and with any number of points if the FWS is  generated by a nonzero Schwartz function, for which the absolute value of the  Fourier transform  is decreasing at infinity. We also prove  the linear independence of  any FWS generated by  a nonzero square integrable function, for which the Fourier transform has certain behavior at infinity. Such a function  can be any  square integrable function that is a linear complex combination of  real valued rational  and exponential functions.
\end{abstract} \maketitle

\section{ Introduction}

Let $\phi$ be a nonzero measurable function on $\mathbb{R}$ and  let $\Lambda=\{(\lambda_k,\beta_k)\}_{k=1}^N \subset (0, \infty) \times \mathbb{R}$,   the Finite Wavelet System (FWS) generated by $\phi$ and parameterized by $\Lambda$ is the set of time-dilation translates of $\phi$
$$\mathcal{W}(\phi, \Lambda)=\{\phi(\lambda_k x-\beta_k)\}_{k=1}^N.$$

On the one hand, there are many examples of non trivial FWS that are linearly dependent in $L^2(\mathbb{R})$ and others in $L^1(\mathbb{R})$. Indeed, examples of   \emph{the two-scale difference equation}  that has the form
\begin{eqnarray}\label{2scale1}
\phi(x) = \sum_{k=0}^{N} c_k \phi(\lambda x-\beta_k )\quad (a.e)
\end{eqnarray}
has solutions in $L^1(\mathbb{R}) \setminus \{0\}$ for some $\lambda >1$. For each integer $K \geq 2$, the lattice two-scale difference equation has the form
\begin{eqnarray}\label{2scalell}
\phi(x) = \sum_{k=0}^{N} c_k \phi(K x- k )\quad (a.e).
\end{eqnarray}

On the other hand,  Christensen and Lindner proved the linear independence of each FWS generated by a nonzero square integrable function, for which the Fourier transform is continuous and compactly supported \cite{Chr}. Thus, they proved the linear independence of FWS generated by each nonzero function in a dense subspace of $L^2(\mathbb{R})$.  Bownik and  Speegle  also proved the linear independence of Parseval wavelets, which can be compared to the linear independence of  infinite rectangular Gabor systems \cite{Bow}. The main purpose of this paper is to present some cases of linearly independent FWS. However, we incidently prove some results on the two-scale difference equation.

In Section 1, we prove a lemma on the order of regularity of a compactly supported $L^1-$solution (if it exists) to Equation \ref{2scale1}. The lemma is then used to prove the independence of each three point FWS generated by a nonzero Schwartz function. Incidently, the lemma is also used to find an upper bound of the order of regularity of solutions to  Equation \ref{2scale1}. This upper bound and an existing lower bound lead to the exact value of that order for several examples in the literature. In addition, we prove that the equation
\begin{eqnarray}\label{Bernolli}
\phi (x)=\frac{\lambda}{2}(\phi(\lambda x- 1) +\phi(\lambda x + 1))
\end{eqnarray}
has  no solution in $ C^n_c(\mathbb{R})\setminus \{0\}$ if $\alpha < 2^{-\frac{1}{n+1}}$, where $\alpha = 1/\lambda$. To the best of our knowledge, this result is not published prior to this paper. Meanwhile, the important Equation \ref{Bernolli} is related to the so called \emph{symmetric Bernoulli convolutions} and  has been studied for several decades with connections to random walk, harmonic analysis, the theory of algebraic numbers, and dynamical systems \cite{Per,Var}.

In Section 2, we prove that $\mathcal{W}(\phi, \Lambda)$ is linearly independent in the following cases:
\begin{enumerate}
  \item[(1)] $\phi$ is with faster than exponential decay  and  the support of $\phi$ is not compact.
 % \item $\phi \in \mathcal{C}^{\infty}_{c}(\mathbb{R})$, the space of infinitely differentiable functions with compact support.
  \item[(2)] $\phi \in \mathcal{S}(\mathbb{R}) \setminus \{0\}$, the space of Schwartz functions, and $card(\Lambda)=3$.
  \item[(3)] $\phi \in \mathcal{S}(\mathbb{R}) \setminus \{0\}$, and $|\widehat{\phi(\gamma)}|$ and $|\widehat{\phi(-\gamma)}|$ are ultimately decreasing.
\end{enumerate}

The proof of Result (1) is pretty easy, while it is not yet proven that a Finite Gabor System (FGS) is linearly independent for all functions with faster than exponential decay. A FGS has the form $\mathcal{G}(g, \Lambda)=\{e^{2\pi i\beta_kx}g(x+\alpha_k)\}_{k=1}^N$. The Heil, Ramanathan,  and  Topiwala (HRT) conjecture states that each FGS generated by a nonzero square integrable function $g$ is linearly independent in $L^2(\mathbb{R})$. The conjecture is still open today even if $g$ is assumed to be in  $\mathcal{S}(\mathbb{R})$ ~\cite{Hei1}. The results of this paper lead to the following conjecture.
\begin{conj}
Each finite wavelet system generated by a nonzero Schwartz function is linearly independent.
\end{conj}

In Section 3,  we prove the linear independence of $\mathcal{W}(\phi, \Lambda)$ if $\phi$ is a nonzero square integrable function for which $\widehat{\phi}$ is a linear complex combination of functions  $\in \mathcal{L}\mathcal{E}(\mathbb{R})$: the space of \emph{logarithmico-exponential} functions.  This space consists of real valued functions that can be expressed using the standard arithmetic operations ($+,-,\times, \div$) along with the identity, constant, exponential, and logarithmic functions. Thus, we prove the linear independence of $\mathcal{W}(\phi, \Lambda)$ if $\phi$ is a nonzero square integrable rational function, $\phi = e^{-n|x|}$ for integers $n>0$,  $ \phi =  e^{-x^2}$, or (for a random example)
$ \widehat{\phi} = \frac{\gamma\ln |\gamma|}{e^\gamma+e^{-\gamma}}$.

\section{A Lemma on The Two-scale Difference Equation}
The result of the follwing lemma is folklore, though we have not been able to locate an explicit proof in the literature.

\begin{lem}\label{2scale}
 If $\phi \in L^1(\mathbb{R}) \setminus \{0\}$ is  a solution with compact support to Equation  \ref{2scale1}, then $|c_0|, |c_N|  < \lambda$. Further, if $|\phi(x)|/|x-a|^{\mu}$ is bounded for almost all $x$ in a neighborhood of $a$, where $\mu \geq 0$ and $ a=(\lambda-1)^{-1}\beta_0$ (resp. $a=(\lambda-1)^{-1}\beta_N$), then $|c_0|  \leq \lambda^{-\mu}$ (resp. $|c_N|  \leq \lambda^{-\mu}$).
\end{lem}
\begin{proof} Suppose $\phi \in L^1(\mathbb{R}) \setminus \{0\}$ is  a solution with compact support to Equation  \ref{2scale1}. If needed we can replace $\phi$ with $T_{\beta}\phi$, where $\beta=(1-\lambda)^{-1}\beta_0$, to assume without loss of generality that $\beta_0=0$. Hence,  $supp(\phi) \subset [0,(\lambda-1)^{-1}\beta_N]$. We then have
\begin{eqnarray}\label{2scale2}
\phi(x)= c_0  \phi(\lambda x), \mbox{ for almost all }  x \in [0,\frac{\beta_1}{\lambda}].
\end{eqnarray}
Therefore
\begin{eqnarray}\label{2scale3}
\int_{0}^{\frac{\beta_1}{\lambda}} | \phi(x)| dx =\frac{|c_0|}{\lambda}  \int_{0}^{\beta_1} | \phi(x)| dx.
\end{eqnarray}
Thus,  $|c_0|\geq \lambda$ implies $\phi(x)=0$ for almost all $x \in [\frac{\beta_1}{\lambda},\beta_1]$. Using Equality \ref{2scale2} $m$ times leads to $\phi(x)=0$ for almost all $x \in [\frac{\beta_1}{\lambda^m},\beta_1]$. Consequently, we have   $\phi(x)=0$ for almost all $x \in [0,\beta_1]$ and Equality \ref{2scale2} holds for almost all  $ x \in [0,\frac{2\beta_1}{\lambda}]$. Using \ref{2scale2} and \ref{2scale3} as before, we obtain $\phi(x)= 0, \mbox{ for almost all }  x \in [0,2\beta_1]$. After enough iterations of the above process, we obtain the contradiction $\phi(x)=0$ almost everywhere.

Now let $\mu \geq 0$ and assume that $|\phi(x)|/|x|^{\mu}$ is bounded for almost all $x$ in  a neighborhood of $0$. Therefore, $|\phi(x)|/|x|^{\mu+\varepsilon} \in L^1(\mathbb{R})$ for all $\varepsilon<1$. Using Equality \ref{2scale2}, we obtain
\begin{eqnarray*}
\int_{0}^{\frac{\beta_1}{\lambda}} \frac{| \phi(x)|}{x^{\mu+\varepsilon}} dx =|c_0|\lambda^{\mu+\varepsilon-1}  \int_{0}^{\beta_1} \frac{| \phi(x)|}{x^{\mu+\varepsilon}}dx.
\end{eqnarray*}
We then proceed as we did in the first part of this proof to deduce  $|c_0|\lambda^{\mu+\varepsilon-1} <1$ for all $\varepsilon<1$. Consequently, we have $|c_0|\lambda^{\mu} \leq 1$.

The function $\varphi(x)=\phi(-x+(\lambda-1)^{-1}\beta_N)$ satisfies the two-scale equation
\begin{eqnarray*}
\varphi(x) = \sum_{k=0}^{N} c_k \varphi(\lambda x-   (\beta_N -\beta_k )) \quad (a.e).
\end{eqnarray*}
Hence, we can complete the proof of Lemma \ref{2scale}.
\end{proof}

The second condition  in Lemma \ref{2scale} can be replaced with the weaker condition $|\phi(x)|/|x-a|^{\mu+\varepsilon} \in L^1(\mathbb{R})$  $\forall \varepsilon <1 $, where $ a=(\lambda-1)^{-1}\beta_0$ or $a=(\lambda-1)^{-1}\beta_N$). Any one of those conditions holds if  $\phi \in C^\mu_c(\mathbb{R})$ when $\mu$ is a nonnegative integer or $\phi$ is H\"{o}lder continuous with exponent  $\mu$, i.e., there is $M>0$ such that
\begin{equation*}
 |\phi(x) -\phi(y)| < M|x-y|^{\mu}, \quad  \forall x,y \in \mathbb{R}.
\end{equation*}
 Hence, we obtain the following corollary, which bounds from upper the degree of regularity of a compactly supported $L^1$-solution to Equation \ref{2scale1}.
\begin{cor}\label{2scalec1} Suppose $\phi  \in L^1(\mathbb{R})\setminus \{0\}$ is  a compactly supported  solution to Equation  \ref{2scale1}. If $\mu \geq 0$, then the following statements hold.
\begin{enumerate}[(a)]
  \item  If $\phi$ is H\"{o}lder continuous with exponent  $\mu$ or $\phi  \in C^\mu(\mathbb{R})$ when $\mu$ is an integer, then
  \begin{eqnarray}\label{2scaleb}
  \mu \leq \frac{\min (-\ln |c_0|,-\ln |c_N|)}{\ln \lambda}.
  \end{eqnarray}
  \item  Equation  \ref{2scale1} has no solution in $C^{\infty}_c(\mathbb{R})\setminus \{0\}$.
\end{enumerate}
\end{cor}

Equation \ref{2scale2} easily implies that $\phi$ must be discontinuous at $a$, where $ a=(\lambda-1)^{-1}\beta_0$ (resp. $a=(\lambda-1)^{-1}\beta_N$), if $|c_0|  \geq 1 $ (resp. $|c_N|  \geq 1$). This fact is noticed by Colella and Heil in the lattice case \cite{Col}. For the lattice case also, Daubechiesf and  Lagarias used the infinite matrix product representations  to bound from below the degree of regularity of a compactly supported $L^1$-solution to Equation \ref{2scalell} \cite{Dau1}. Curiously, the upper bound computed by Corollary \ref{2scalec1} coincides with the optimal degree of regularity for each example listed in \cite{Dau1}. For example, the exact degree of regularity of the Rham function that is the normalized $L^1$-solution to the equation
\begin{eqnarray*}
\phi(x)= \phi(3x)+\frac{1}{3}[\phi(3x-1)+\phi(3x+1)]+\frac{2}{3}[\phi(3x-2)+\phi(3x+2)]
\end{eqnarray*}
is equal to the upper bound computed using Corollary \ref{2scalec1}:
\begin{eqnarray*}
\frac{-\ln(2/3)}{\ln 3}\simeq 0.36907024642.
\end{eqnarray*}

We finish this section with the following corollaries for the case $N=2$.

\begin{cor}\label{2scalec2} Suppose Equation  \ref{2scale1} with $N=2$ has a compactly supported solution $\phi \in L^1(\mathbb{R}) \setminus \{0\}$. Then the following statements hold.
\begin{enumerate}[(a)]
  \item If $\lambda > 2$,  then $\phi$ is not bounded.
  \item If $\lambda = 2$ and $\phi$ is bounded, then $c_0=c_1=1$.
  \item If $1<\lambda < 2$ and $\phi$ is H\"{o}lder continuous with exponent  $\mu$ or $\phi  \in C^\mu(\mathbb{R})$ when $\mu$ is an integer, then
  \begin{eqnarray}\label{2scaleb}
  \mu \leq \frac{1}{\log_2 \lambda}-1.
  \end{eqnarray}
 % \item If $\phi  \in C^n_c(\mathbb{R})\setminus \{0\}$, then $\lambda^{n+1}  \leq 2 $.
\end{enumerate}
\end{cor}

The case $N=2$ includes Equation \ref{Bernolli} that  has always a unique non trivial normalized solution $f_\alpha$, where $ \alpha=1/\lambda$, that is a compactly supported measure \cite{Per,Var}. Finding all values of $ \alpha \in (\frac{1}{2},2)$, for which the measure $f_\alpha$ is absolutely continuous is still today an open problem.   Erd\H{o}s proved for every integer $n\geq 0$ there is $a(n)<1$ such that for almost all $\alpha \in (a(n),1)$, $f_{\alpha} \in C^n_c(\mathbb{R}) \setminus \{0\}$ \cite{Erd}. Erd\H{o}s also proved that there are $\alpha$ for which the measure $f_\alpha$ is not absolutely continuous.  Another breakthrough result is due to B. Solomyak, who proved  for almost all $\alpha \in (\frac{1}{2},1)$, $f_{\alpha} \in L^2(\mathbb{R}) \setminus \{0\}$ and for almost all $\alpha \in (\frac{1}{\sqrt{2}},1)$, $f_{\alpha}$ is continuous \cite{Per}. However, it is not yet known if there exists $\alpha \in (\frac{1}{2},\frac{1}{\sqrt{2}})$ for which $f_{\alpha}$ is continuous.  Using  Corollary \ref{2scalec2}.d, we deduce the following result.
\begin{cor}  Equation \ref{Bernolli} has no solution $f_{\alpha}  \in C^n_c(\mathbb{R})\setminus \{0\}$ if $\alpha < 2^{-\frac{1}{n+1}}$. \end{cor}

Therefore, there is no $\alpha \in (\frac{1}{2},\frac{1}{\sqrt{2}})$ for which $f_{\alpha}$ is $C^1$.

\section{FWS generated by Schwartz Functions}
We say that a measurable function $\phi$ is with faster than polynomial decay if
$$\lim_{x \rightarrow \pm \infty} x^n\phi(x)=0, \quad  \forall n >0.$$

We say that a measurable function $\phi$ is with faster than exponential decay if
$$\lim_{x \rightarrow \infty} e^{tx}\phi(x)=\lim_{x \rightarrow -\infty} e^{-tx}\phi(x)=0, \quad  \forall t >0.$$

If it exists, $\phi^{(n)}$ will denote the $n^{th}$ derivative of $\phi$.
\begin{lem}\label{smallbig}
Let $\phi \in L^{1}(\mathbb{R}) \setminus \{0\}$. Let $ \Lambda=\{(\lambda_{k},\beta_{k})\}_{k=0}^{N} \subset (0,\infty) \times \mathbb{R}$.  Then $\mathcal{W}(\phi, \Lambda)$ is linearly independent in the following cases.
\begin{enumerate}[(a)]
\item  $\phi$ is with faster than exponential decay and  the support of $\phi$ is not compact.
  \item  $\phi$ is with faster than polynomial decay, the support of $\phi$ is not compact, and $\lambda_k<\lambda_{0}$ for $k=1,\dots,N$.
  \item  $\phi \in \mathcal{C}^{\infty}(\mathbb{R})$, $\phi^{(n)} \in L^{1}(\mathbb{R}) \setminus \{0\}$ for each $n>0$,  and $\lambda_k>\lambda_{0}$ for $k=1,\dots,N$.
\end{enumerate}
\end{lem}

\begin{proof}  Suppose $\mathcal{W}(\phi, \Lambda)$ is linearly dependent. We may assume  without loss of generality there are $c_1,\dots,c_N \in \mathbb{C} $ such that
\begin{eqnarray}\label{smallbig1}
 \phi(x)= \sum_{k=1}^{M}  c_k \phi(x-\beta_k)+ \sum_{k=M+1}^{N}  c_k \phi(\lambda_k x-\beta_k),
\end{eqnarray}
where $\lambda_{0}=1$, $\lambda_{k} \neq 1 $ for $k=M+1, \dots, N$,  $\beta_{0}=0$, $\beta_{1},\dots, \beta_M <0$, and $\beta_{M+1},\dots, \beta_N \in \mathbb{R}$.

For the first two cases, we may also assume that $\lambda_{M+1}, \dots, \lambda_N>1$.

\emph{(a)}  The fact that the support of $\phi$ is not compact leads to two possibilities:
\begin{eqnarray*}
0<\int_A^\infty |\phi(x)e^{tx}| dx < \infty, \quad \forall t,A>0 \mbox{ or } 0<\int_{-\infty}^A |\phi(x)e^{tx}| dx < \infty, \quad\forall t,A<0.
\end{eqnarray*}
Suppose the first possibility holds. Let $ t>0$ and let
$A$ be big enough to have $\lambda_k A -\beta_k \geq A, \forall k=M+1,\dots, N$. We have
%\begin{eqnarray*}
%\int_A^\infty |\phi(x)e^{tx}| dx &\leq &  \sum_{k=1}^{M}  |c_k|e^{t\beta_k }\int_A^\infty |\phi( x-\beta_k)e^{t(x-\beta_k)}| dx\\
%&+& \sum_{k=M+1}^{N}  |c_k| \int_A^\infty|\phi(\lambda_k x-\beta_k)e^{t(\lambda_k x-\beta_k)}|e^{t[(1-\lambda_k) x+\beta_k]}dx,
%\end{eqnarray*}
%and so
\begin{eqnarray*}
\int_A^\infty |\phi(x)e^{tx}| dx &\leq &  \sum_{k=1}^{M}  |c_k|e^{t\beta_k }\int_A^\infty |\phi(x)|e^{tx} dx\\
&+& \sum_{k=M+1}^{N}  |c_k| e^{t[(1-\lambda_k) A+\beta_k]}\int_A^\infty |\phi(x)e^{tx}| dx.
\end{eqnarray*}
Simplifying the last inequality and letting  $t \rightarrow \infty$ lead to the contradiction $1 \leq 0$.

Similarly, we can obtain a contradiction if the second possibility holds.

\emph{(b)}  In this case $M=0$ and the right side of Equality \ref{smallbig1} does not includes the first sum. If needed, we can replace $\phi$ with $T_\beta \phi$, where $\beta$ is such that $\beta_k+(\lambda_k-1)\beta< 0, \forall k=1,\dots, N$, to assume without loss of generality that $\beta_k< 0, \forall k=1,\dots, N$. The fact that the support of $\phi$ is not compact leads to two possibilities:
\begin{eqnarray*}
0<\int_A^\infty |\phi(x)x^n| dx < \infty, \quad \forall A,n>0 \mbox{ or } 0<\int_{-\infty}^{-A} |\phi(x)x^n| dx < \infty, \quad\forall A, n>0.
\end{eqnarray*}

Suppose the first possibility holds. We have
\begin{eqnarray*}
 \int_{0}^{\infty}|\phi(x)x^n|  dx \leq  \sum_{k=1}^{N} \frac{|c_k|}{\lambda_k^{n}}  \int_{0}^{\infty} |\phi(\lambda_k x-\alpha_k )(\lambda_k x-\beta_k)^n| (\frac{\lambda_k x}{\lambda_k x-\beta_k})^n dx
\end{eqnarray*}
and so
\begin{eqnarray}\label{smallbig2}
 \int_{0}^{\infty}|\phi(x)x^n|  dx \leq \sum_{k=1}^{N} \frac{|c_k|}{\lambda_k^{n+1}}  \int_{0}^{\infty} |\phi(x)x^n|dx
\end{eqnarray}
Simplifying Inequality \ref{smallbig2} and Letting $n \rightarrow \infty$ lead to the contradiction $1 \leq 0$.

Similarly, we can find a contradiction if the second possibility holds.

The proof of statement \emph{(c)} is straightforward.

\end{proof}

\begin{cor}\label{three}
Every three point FWS generated by a nonzero Schwartz function is  linearly independent.
\end{cor}
\begin{proof} let $\mathcal{W}=\{\phi(\lambda_k x-\beta_k)\}_{k=1}^3$, where $\phi \in \mathcal{S}(\mathbb{R}) \setminus \{0\}$. The only case that is not following Lemma \ref{smallbig} or Corollary \ref{2scalec1}  is the almost obvious case when $\lambda_1=\lambda_2=\lambda_3$.
\end{proof}

\begin{lem} \label{compact}
Let $\phi \in L^2(\mathbb{R}) \setminus \{0\}$. Every FWS generated by $\phi$ is linearly independent if one of the following statement is satisfied.
\begin{enumerate}[(i)]
 \item There is an $\varepsilon>0$ such that $\widehat{\phi}(\gamma)=0, \forall \gamma \in (-\varepsilon,\varepsilon) $.
  \item $\widehat{\phi}$  is compactly supported.
\end{enumerate}
\end{lem}
\begin{proof} \emph{(i)} Suppose there is a FWS generated by $\phi$ that is linearly dependent.   WLG, we may assume  there are $b_1, \dots, b_M,\beta_1,\dots,\beta_N \in \mathbb{R}$, $ \lambda_{1}, \dots, \lambda_N >1 $, and $a_1, \dots, a_M, c_1,\dots,c_N \in \mathbb{C} \setminus \{0\}$ such that
\begin{eqnarray*}
  p(\gamma) \widehat{\phi}(\gamma )= \sum_{k=1}^{N} c_k e^{2 \pi i \beta_k \gamma}\widehat{\phi}(\frac{\gamma}{\lambda_k}) \quad (a.e), \mbox{ where }  p(\gamma)= \sum_{k=1}^{M} a_k e^{2 \pi i b_k \gamma}.
\end{eqnarray*}
%for almost all $\gamma \in \mathbb{R}$.
%Let $E$ be the set of $\gamma$ for which (1.1) is not satisfied. Then the measure of $E$ is zero. Let    $\phi_0$ be defined as follows:
%\begin{eqnarray*}
%\phi_0(\gamma)=\phi_0(\frac{\gamma}{\lambda_1})= \dots = \phi_0(\frac{\gamma}{\lambda_N})=0, \quad \forall \gamma \in E \sqcup \bigsqcup_{k=1}^{N} \frac{E}{\lambda_k}
%\end{eqnarray*}
%and $\phi_0(\gamma)=\phi(\gamma)$ elsewhere. Replacing $\phi$ with $\phi_0$, we can assume WLG that (1.1) holds for all $\gamma \in \mathbb{R}$.

Let $\delta>0$, let $\lambda = \min \{\lambda_{1}, \dots, \lambda_N \}$, and let $n$ be such that $0<\delta < \lambda^n\varepsilon$. For  almost all $\gamma  \in (-\delta,\delta)$, we have
\begin{eqnarray*}
  (\prod_{j=1}^{N} p(\frac{\gamma}{\lambda_j})) p(\gamma) \widehat{\phi}(\gamma )= \sum_{k=1}^{N} c_k e^{2 \pi i \beta_k \gamma}(\prod_{j=1, j \neq k }^{N} p(\frac{\gamma}{\lambda_j}))\sum_{i=1}^{N}c_i e^{2 \pi i \beta_i\lambda_k \gamma}\widehat{\phi}(\frac{\gamma}{\lambda_i\lambda_k})
\end{eqnarray*}
and $|\gamma/ (\lambda_i\lambda_k )| < \delta / \lambda^2$, for all $1\leq i,k \leq N$. After $n$ iterations of the last process, we can find a trigonometric polynomial $q$ and an equality in the form
\begin{eqnarray*}
  q(\gamma) \widehat{\phi}(\gamma )= \sum_{k=1}^{L} d_k e^{2 \pi i \alpha_k \gamma}\widehat{\phi}(\frac{\gamma}{\mu_k}) \quad (a.e)
\end{eqnarray*}
such that $|\gamma/\mu_k |< \delta / \lambda^n < \varepsilon $, for all $k=1,\dots, L$. Using Lemma 1.1, we deduce that $\widehat{\phi}(\gamma )=0$ for  all $\gamma  \in (-\delta,\delta)$. Since $\delta$ is an arbitrary positive number, we deduce the contradiction that $\phi(\gamma)=0$ almost everywhere.

The proof of Statement \emph{(ii)} can be obtained using similar to the steps in the proof of Statement \emph{(i)}. With the additional condition that $\widehat{\phi}$ is continuous, Statement (ii) is proved in \cite{Chr}.
\end{proof}

A measurable function $\phi$ is said to ultimately have a property $\mathcal{P}$ if there is $A>0$ such that $\phi(x)$ has the property $\mathcal{P}$ for all $x>A$.

\begin{thm}\label{Schwartz}
Let $\phi \in \mathcal{S}(\mathbb{R}) \setminus \{0\}$ be such that $|\ \widehat{\phi}(\gamma)  |\ $ and $|\ \widehat{\phi}(-\gamma)  |\ $ are ultimately decreasing. Every FWS generated by $\phi$ is linearly independent.
\end{thm}
\begin{proof} Assume  $\mathcal{W}(\phi, \Lambda)$  is linearly dependent in $L^2(\mathbb{R})$ for some finite subset $\Lambda$ of $(0.\infty) \times \mathbb{R}$. We may assume without loss of generality there are $c_1,\dots,c_N \in \mathbb{C} \setminus \{0\}$, $\beta_1,\dots,\beta_N \in \mathbb{R}$, $\beta_1,\dots,\beta_M$ are 2 by 2 distinct,  and
$\lambda_1,\dots,\lambda_N> 1$ such that
\begin{eqnarray*}
\sum_{k=1}^M c_k e^{2 \pi i \beta_k \gamma}\widehat{\phi}(\gamma)= \sum_{k=M+1}^{N} c_k e^{2 \pi i \beta_k \gamma}\widehat{\phi}(\lambda_k \gamma).
\end{eqnarray*}

Since the case $M=1$ follows using Proposition  \ref{smallbig}, we assume that $M >1$.

Since $|\ \widehat{\phi}(\gamma)  |\ $ and $|\ \widehat{\phi}(-\gamma)  |\ $ are  ultimately decreasing, then  $\widehat{\phi}$ is compactly supported, in which case $\mathcal{W}(\phi, \Lambda)$ is independent by Lemma \ref{compact}, or either $| \widehat{\phi}(\gamma)  |$ or  $|\ \widehat{\phi}(-\gamma)  |\ $ is ultimately positive. Let's  suppose that $| \widehat{\phi}(\gamma)  |$ is ultimately positive.

Let $\lambda = \min \{\lambda_{M+1},\dots,\lambda_N\}$. We claim the following:
\begin{eqnarray*}
\forall m>0, \exists n_m > m, \forall \gamma \in (1, \sqrt[4]{\lambda}), \quad | \widehat{\phi}(\gamma \lambda^{\frac{n_m+2}{2}})| \leq  \frac{1}{m}|\widehat{\phi}(\gamma \lambda^{\frac{n_m}{2}})|.
\end{eqnarray*}

If the claim is not true, then
\begin{eqnarray*}
\exists m>0, \forall   n> m, \exists \gamma_n \in (1, \sqrt[4]{\lambda}), \quad |\widehat{\phi}(\gamma_n \lambda^{\frac{n+2}{2}})| > \frac{1}{m}|\widehat{\phi}(\gamma_n \lambda^{\frac{n}{2}})|.
\end{eqnarray*}
We can also require that $m$ is big enough to have $|\widehat{\phi}(\gamma_n \lambda^{\frac{n}{2}})| >0$, for all $n>m$.  Therefore, for $n_0 >m $ and for $l>0$ , we have the inequality
\begin{eqnarray*}
\sum_{n\geq n_0} \gamma_n^l \lambda^{\frac{nl+2l}{2}} |\widehat{\phi}(\gamma_n \lambda^{\frac{n}{2}})| > \frac{1}{m}\sum_{n\geq n_0} \gamma_n^l \lambda^{\frac{nl+2l}{2}} |\widehat{\phi}(\gamma_n \lambda^{\frac{n}{2}})|,
\end{eqnarray*}
for which each member is a convergent series because $\widehat{\phi} \in \mathcal{S}(\mathbb{R})$.  The last inequality can also be written as
\begin{eqnarray}
\sum_{n\geq n_0} \gamma_n^l \lambda^{\frac{nl+2l}{2}} |\widehat{\phi}(\gamma_n \lambda^{\frac{n+2}{2}})| &>& \frac{1}{m} \gamma_{n_0}^l \lambda^{\frac{n_0l+2l}{2}} |\widehat{\phi}(\gamma_{n_0} \lambda^{\frac{n_0}{2}})| \\ &+&
\frac{\lambda^{\frac{l}{4}}}{m}\sum_{n\geq n_0} \gamma_{n+1}^l \lambda^{\frac{2nl+5l}{4}} |\widehat{\phi}(\gamma_{n+1} \lambda^{\frac{n+1}{2}})|. \nonumber
\end{eqnarray}

 The fact that $\gamma_n \in (1, \sqrt[4]{\lambda})$ yields the followings:
\begin{eqnarray}
%\lambda^{-1/4} &\leq \frac{\gamma_{n+1} }{\gamma_n }  \leq& \lambda^{1/4}\\
 \frac{\gamma_{n+1}^l \lambda^{\frac{2nl+5l}{4}}}{\gamma_n^l \lambda^{\frac{nl+2l}{2}}} &=& \frac{\gamma_{n+1}^l }{\gamma_n^l }\lambda^{\frac{l}{4}} > 1 \\
  \frac{\gamma_{n+1} \lambda^{\frac{n+1}{2}}}{\gamma_n \lambda^{\frac{n+2}{2}}} &=& \frac{\gamma_{n+1} }{\gamma_n }\lambda^{-\frac{1}{2}} < \lambda^{-\frac{l}{4}} <1
\end{eqnarray}

Let $A>0$ such that $|\widehat{\phi}(\gamma) |$ is decreasing for all $\gamma >A$.  Let $n_0$ be large enough to have $\lambda^{\frac{n_0}{2}} > A$. Therefore, (2.1)-(2.3) imply the inequality
\begin{eqnarray*}
(1 - \frac{\lambda^{\frac{l}{4}}}{m})\sum_{n\geq n_0} \gamma_n^l \lambda^{\frac{nl+2l}{2}} |\widehat{\phi}(\gamma_n \lambda^{\frac{n+2}{2}})| &>& \frac{1}{m} \gamma_{n_0}^l \lambda^{\frac{n_0l+2l}{2}} |\widehat{\phi}(\gamma_{n_0} \lambda^{\frac{n_0}{2}})|.
\end{eqnarray*}
Taking $l$ large enough makes negative the left hand side of  the last inequality, which yields a contradiction. Hence, the claim is true, and so
\begin{eqnarray*}
\forall \gamma \in (1, \sqrt[4]{\lambda}), \quad \lim_{m \rightarrow \infty} \frac{\widehat{\phi}(\gamma \lambda^{\frac{n_m+2}{2}})}{\widehat{\phi}(\gamma \lambda^{\frac{n_m}{2}})} = 0.
\end{eqnarray*}

Since $\lambda = \min \{\lambda_{M+1},\dots,\lambda_N\}$ and $|\widehat{\phi}|$ is ultimately decreasing, we obtain
\begin{eqnarray*}
 \forall \gamma \in (1, \sqrt[4]{\lambda}), \quad &&\lim_{m \rightarrow \infty} |\sum_{k=1}^M c_k e^{2 \pi i \beta_k \gamma \lambda^{\frac{n_m}{2}} }| \leq  \lim_{m \rightarrow \infty} \sum_{k=M+1}^{N} |c_k \frac{\widehat{\phi}(\gamma\lambda_k \lambda^{\frac{n_m}{2}})}{\widehat{\phi}(\gamma \lambda^{\frac{n_m}{2}})}|\\
  &\leq& \lim_{m \rightarrow \infty} \sum_{k=M+1}^{N} |c_k \frac{\widehat{\phi}(\gamma\lambda^{\frac{n_m+2}{2}})}{\widehat{\phi}(\gamma \lambda^{\frac{n_m}{2}})}|=0
\end{eqnarray*}
Consequently,
\begin{eqnarray*}
 \forall \gamma \in (1, \sqrt[4]{\lambda}), \quad \lim_{m \rightarrow \infty} \sum_{k=1}^{M-1} c_k e^{2 \pi i (\beta_k - \beta_M)\gamma \lambda^{\frac{n_m}{2}} }=  - c_M.
\end{eqnarray*}
Therefore, using Lebesgue's dominated convergence theorem, we have
\begin{eqnarray*}
 c_M(1-\sqrt[4]{\lambda})&=& \lim_{m \rightarrow \infty} \sum_{k=1}^{M-1} \int_{1}^{\sqrt[4]{\lambda}}c_k e^{2 \pi i (\beta_k - \beta_M)\gamma \lambda^{\frac{n_m}{2}} }  d\gamma   \\
   a_M(1-\sqrt[4]{\lambda})&=&\lim_{m \rightarrow \infty} \sum_{k=1}^{M-1} c_k  \frac{e^{2 \pi i (\beta_k - \beta_M)\lambda^{\frac{2n_m+1}{4}} } -e^{2 \pi i (\beta_k - \beta_M)\lambda^{\frac{n_m}{2}} } }{2 \pi i (\beta_k - \beta_M) \lambda^{\frac{n_m}{2}}}=0,
\end{eqnarray*}
and so we obtain the contradiction $a_M=0$.
\end{proof}

Similar to the result of Theorem \ref{Schwartz} is not known for FGS. A theorem in \cite{Ben} proves the independence of each four point FGS generated by a nonzero $L^2$-function $g$ such that $g(x)$ and $g(-x)$ are ultimately positive and decreasing.

\section{FWS generated by  Functions with certain behaviour at Infinity}
Let $\mathcal{F}$ be a space of real valued functions $f$ that are ultimately defined, i.e., $f$ is everywhere defined on an interval $(A, \infty)$ for some real number $A$. We define on $\mathcal{F}$ the equivalence relation $\sim$  by writing $f \sim g$ to mean $f(x)$ and $g(x)$ are ultimately equal. The equivalence class associated with $f \in \mathcal{F}$ is denoted by $germ(f)$. Addition and multiplication of functions are compatible with respect to $\sim$, and so the set $\overline{\mathcal{F}}= \{ germ(f): f \in \mathcal{F}\} $ is a commutative ring. An example of $\mathcal{F}$ is the space $\mathcal{LE}(\mathbb{R})$ of  logarithmico-exponential functions. Some properties of $\mathcal{LE}(\mathbb{R})$ are listed in the following proposition.
\begin{prop}[\cite{Har}]\label{hardy} Let $\overline{\mathcal{LE}(\mathbb{R})}=\{ germ(f): f \in \mathcal{LE}(\mathbb{R}) \}$
\begin{enumerate}
  \item Every Logarithmico-exponential function is ultimately  analytic.
  \item If $f,g \in \mathcal{LE}(\mathbb{R})$ such that $germ(f) \neq 0$ and $germ(g) \neq 0$, then the limit at infinity of $f/g$ or $g/f$ is finite.
If the limit at infinity of $f/g$ is finite we say that $f$ is asymptotically smaller than $g$ and we write $germ(f) \preceq germ(g)$.
 \item $\overline{\mathcal{LE}(\mathbb{R})}$  is a Hardy field, i.e, it is a field that is closed by differentiation and it is well-ordered with respect to the relation $\preceq$ .
\end{enumerate}
\end{prop}

The proof of the following theorem is following the same steps used to prove a similar theorem on FGS \cite{Ben}. Nevertheless, we have to check the points that make different the proof of the new theorem.

\begin{thm}\label{hardyfield} Let $\phi \in L^2(\mathbb{R}) \setminus \{0\}$ such that $\widehat{\phi}(\gamma)$  is a complex linear combination of square integrable functions whose germs $ \in \overline{\mathcal{LE}(\mathbb{R})}$. Then every non trivial FWS generated by $\phi$  is linearly independent.
\end{thm}
\begin{proof}    For the sake of simplicity  we assume that $\widehat{\phi}(\gamma) \in \mathcal{LE}(\mathbb{R})$, since the generalization is straightforward. Suppose there is a linearly dependent  FWS generated by $\phi$. We may assume without loss of generality that
\begin{eqnarray}\label{hardyfield1}
\sum_{k=1}^M c_k e^{2 \pi i \beta_k \gamma}\widehat{\phi}(\gamma)= \sum_{k=M+1}^{N} c_k e^{2 \pi i \beta_k \gamma}\widehat{\phi}(\lambda_k \gamma),
\end{eqnarray}
where $c_1,\dots,c_N \in \mathbb{C}$, $\beta_1,\dots,\beta_N \in \mathbb{R}$, $\beta_1,\dots,\beta_M$ are distinct,  and
$\lambda_1,\dots,\lambda_N> 1$.

\emph{(i)} It is known that $\widehat{\phi}$ is ultimately analytic \cite{Har}, i.e., analytic on an interval
 $(A, \infty)$, for some $A>0$. Let $I$ be a bounded open interval $\subset (0,\infty)$. Iterating \ref{hardyfield1} enough times, we can write
 \begin{eqnarray*}
p(\gamma)\widehat{\phi}(\gamma)= \sum_{k=1}^{n} a_k e^{2 \pi i \beta_k' \gamma}\widehat{\phi}(\lambda_k' \gamma),
\end{eqnarray*}
where $p(\gamma)$ is a trigonometric polynomial and $\lambda_k'$ are big enough to have $\lambda_k' \gamma >A$, for all $\gamma \in I$. Therefore, $p(\gamma)\widehat{\phi}(\gamma)$ is analytic on $I$, and so $\widehat{\phi}(\gamma)=f(\gamma)/p(\gamma)$, where $f$ is analytic on $(0,\infty)$. Since $p(\gamma)$ is also analytic, for each $\gamma_0 >0$, there is an open interval $I_0$ containing $\gamma_0 $ and an integer $n$ such that
\begin{eqnarray}
  \forall \gamma \in I_0, \quad \widehat{\phi}(\gamma)= \frac{f(\gamma)}{p(\gamma)} = (\gamma-\gamma_0)^n h(\gamma),
\end{eqnarray}
where $h$ is analytic and never vanishes on $I_0$. The fact that $\widehat{\phi} \in L^2(\mathbb{R})$ implies $n \geq 0$, and so $\widehat{\phi}$ is analytic on $I_0$. We then conclude that $\widehat{\phi}$ is analytic on $(0,\infty)$. Similarly, we can prove that $\widehat{\phi}$ is analytic on $(-\infty,0)$. Thus, in the following we assume that $\widehat{\phi}$ is analytic on $\mathbb{R} \setminus \{0\}$

{\it ii. }After relabeling, we may suppose that
\begin{eqnarray}\label{hardyfield2}
  \sum_{k=1}^N e^{2\pi i \beta_k \gamma} f_k(\gamma) =0, \quad \mbox{where} \quad  f_k(\gamma) = \sum_{n=1}^{N_k} c_{(k,n)}\widehat{\phi}(\lambda_{(k,n)}\gamma),
\end{eqnarray}
where $\beta_1,\ldots, \beta_N \in \mathbb{R}$ are distinct and
 $c_{(k,1)}, c_{(k,2)},\ldots,c_{(k,N_k)} \in \mathbb{C} \setminus \{ 0 \}$ and $\lambda_{(k,1)},\ldots,\lambda_{(k,N_k)}>0$, for each $k=1,\ldots, N$.

 Let $k \in \{1,\ldots, N\}$ and assume that $germ(f_k)=0$. Since $f_k$ is analytic on $(0,\infty)$, then $f_k$ is equal to zero on $(0,\infty)$. But $e^{\frac{t}{2}}f_k(e^t)$ is a sum of time translates of  $e^{\frac{t}{2}}\widehat{\phi}(e^t) \in L^2(\mathbb{R})$, and so $\widehat{\phi}(\gamma)=0, \forall \gamma \in (0,\infty)$. Similarly, if $f_k(-\gamma)$ is ultimately equal to zero, then  $\widehat{\phi}(\gamma)=0, \forall \gamma \in (-\infty,0)$. We conclude there is at least one $k \in \{1,\ldots, N\}$ such that $germ(f_k)\neq 0$ or $germ(\check{f}_k)\neq 0$, where $\check{f}_k(\gamma)= f_k(-\gamma)$.

{\it iii. } Without loss of generality, we may assume there is at least one $k \in \{1,\ldots, N\}$ such that $germ(f_k)\neq 0$. Using this and Proposition \ref{hardy}, there is $k_0 \in \{1,\ldots, N\}$ such that $germ(f_{k_0})\succeq germ(f_{k})$ for all $k \in \{1,\ldots, N\}$ and $germ(f_{k_0})\neq 0$.

Now, let $(\gamma_n)$ be a sequence converging to infinity such that
\begin{eqnarray*}
 \forall k \in \{1,\ldots, N\}, \quad \lim_{n \rightarrow \infty} e^{2 \pi i \beta_k \gamma_n} = L_k \in \mathbb{C}\setminus \{0\}.
\end{eqnarray*}   Therefore,
\begin{eqnarray}
\lim_{n \rightarrow \infty}  \sum_{k=1}^N e^{2\pi i \beta_k (\gamma+\gamma_n)}  \frac{f_k(\gamma+\gamma_n)}{ f_{k_0}(\gamma+\gamma_n)} = \sum_{k=1}^N L_kl_k e^{2\pi i \beta_k \gamma} =0,
\end{eqnarray}
where $l_k \in \mathbb{R}$ and $l_{k_0}=1$. Thus, the last quality in 4.4 is a contradiction.
\end{proof}

A generalization of Theorem \ref{hardyfield} can be stated if $\mathcal{LE}(\mathbb{R})$ is replaced with $\mathcal{F}$, a space of real valued functions that are ultimately defined such that $\overline{\mathcal{F}}$ is a subset of a Hardy field closed under dilations with one of the following requirements:
\begin{itemize}
  \item Each one of the function $\widehat{\phi}(\gamma)$  and $\widehat{\phi}(-\gamma)$  is a complex linear combination of square integrable ultimately analytic functions whose germs are in $\overline{\mathcal{F}}$.
  \item  The function $\widehat{\phi}$  is a complex linear combination of square integrable functions that are analytic at zero and whose germs are in $\overline{\mathcal{F}}$.
\end{itemize}

Similar to Theorem \ref{hardyfield} for FGS implies the independence of each FGS generated by a nonzero square integrable logarithmo-exponential function or by its Fourier transform \cite{Ben}. Meanwhile, Theorem \ref{hardyfield} only implies the linear independence of each FWS generated by a nonzero square integrable function whose Fourier transform is a logarithmo-exponential function. 

\section*{Acknowledgements}
The author gratefully thanks Chris Heil for instructive discussions that dramatically  improved this paper.
The author also gratefully thanks P\'{e}ter  Varj\'{u} and Boris Solomyak for kindly replying to questions about some results in this paper.

\bibliographystyle{amsplain} % Harvard in google Scholar

\end{document}